\documentclass[11pt,reqno]{amsart}
\usepackage{amsmath}
\usepackage{amscd,amsfonts,amsthm,amssymb,latexsym,mathrsfs}
\usepackage{url,color}

\usepackage{hyperref}
\hypersetup{
    bookmarks=true,         
    unicode=false,          
    pdftoolbar=true,        
    pdfmenubar=true,        
    pdffitwindow=false,     
    pdfstartview={FitH},    
    pdftitle={My title},    
    pdfauthor={Author},     
    pdfsubject={Subject},   
    pdfcreator={Creator},   
    pdfproducer={Producer}, 
    pdfkeywords={keyword1} {key2} {key3}, 
    pdfnewwindow=true,      
    colorlinks=true,       
    linkcolor=blue,          
    citecolor=blue,        
    filecolor=blue,      
    urlcolor=black    
    }

\theoremstyle{plain}
   \newtheorem{theorem}{Theorem}[section]
   \newtheorem{proposition}[theorem]{Proposition}
   \newtheorem{lemma}[theorem]{Lemma}

   \newtheorem{conjecture}[theorem]{Conjecture}
   \newtheorem{question}{Question}
\theoremstyle{definition}

\theoremstyle{remark}

\def\kkk{\kern.2ex\mbox{\raise.5ex\hbox{{\rule{.35em}{.12ex}}}}\kern.2ex}

\numberwithin{equation}{section}
\newcommand{\xx}{\mathbf{x}}
\newcommand{\yy}{\mathbf{y}}

\newcommand{\ww}{\mathbf{w}}

\newcommand{\NN}{\mathbb{N}}

\newcommand{\one}{\mathbf{1}}

\newcommand{\A}{\mathcal{A}}

\newcommand{\LL}{\mathcal{L}}
\renewcommand{\SS}{\mathbb{S}}

\newcommand{\RR}{\mathbb{R}}

\newcommand{\sym}{\mathfrak{S}}

\renewcommand{\Im}{{\rm Im}}

\def\newop#1{\expandafter\def\csname #1\endcsname{\mathop{\rm
#1}\nolimits}}

\newop{per}
\newop{diag}
\newop{supp}
\newop{Proj}
\newop{span}
\newop{rank}
\newop{Sym}
\newop{sign}
\newop{Int}
\newop{disc}
\newop{mult}
\newop{St}
\newop{glb}
\newop{MAP}
\newop{mesh}

\title[Lorentzian and real stable polynomials and Euclidean balls]{Spaces of  Lorentzian and real stable polynomials are Euclidean balls}


\author{Petter Br\"and\'en}
\thanks{The author is a Wallenberg Academy Fellow supported by the Knut and Alice Wallenberg foundation and the G\"oran Gustafsson foundation}

\address{Department of Mathematics, KTH, Royal Institute of Technology, Stockholm, Sweden.}
\email{pbranden@kth.se}

\begin{document}

\maketitle
\begin{center}{\small \emph{Till Danne, min v\"an}}
\end{center}

\begin{abstract}
We prove that  projective spaces of  Lorentzian and real stable polynomials are homeomorphic to closed Euclidean balls. This solves a conjecture of June Huh and the author. The proof utilizes and refines a connection between the symmetric exclusion process in  Interacting Particle Systems and the geometry of polynomials. 
\end{abstract}
\thispagestyle{empty}

\section{Introduction}
Over the past two decades there has been a surge of activity in the study of stable, hyperbolic and Lorentzian polynomials. Spectacular applications in several different areas have been given, see \cite{ALOGV,ALOGV3,BH,MSS,Wag} and the references therein. In this paper we are interested in the shapes of spaces of such polynomials. 

The space of Lorentzian polynomials was studied by Huh and the author in \cite{BH}, by Gurvits in \cite{Gur} who called it the space of strongly log-concave polynomials, and by Anari \emph{et al.} in \cite{ALOGV3} who called it the space of completely log-concave polynomials. This space contains all volume polynomials of convex bodies and projective varieties, as well as homogeneous stable polynomials with nonnegative coefficients. 
The theory  of Lorentzian polynomials links discrete and continuous notions of convexity. Denote by $\mathbb{P}\mathrm{L}^d_n$ the projective space of all Lorentzian degree $d$ polynomials in the variables $w_1, \ldots, w_n$. This remarkable space is stratified by the family of all $\mathrm{M}$-convex sets \cite{Mu} on the discrete simplex.  It was proved in \cite{BH} that the space $\mathbb{P}\mathrm{L}^d_n$ is compact and contractible, and conjectured that  $\mathbb{P}\mathrm{L}^d_n$ is homeomorphic to a closed Euclidean ball. We prove this conjecture here (Theorem~\ref{main})  by utilizing and further developing a powerful connection between the \emph{symmetric exclusion process} (SEP)  and the \emph{geometry of polynomials}, which was discovered and studied  in \cite{BBL}. The proof technique also applies to spaces of real stable polynomials, and we prove that these spaces are homeomorphic to closed Euclidean balls. 

It was proved in \cite{BBL} and \cite{BH} that SEP preserves the classes of multiaffine stable and multiaffine Lorentzian polynomials, respectively. In Sections \ref{bulk} and \ref{stabilt} we study in more detail how SEP acts on multiaffine Lorentzian polynomials. We prove that SEP is a contractive flow (Lemma~\ref{shrink}) that maps the boundaries of the spaces into their interiors for all positive times (Lemma~\ref{push}). Furthermore, SEP contracts the spaces to a  point in the interior. These results are used in Section~\ref{proof}, in conjunction with a construction of Galashin, Karp and  Lam \cite{GKL}, to prove that various spaces of Lorentzian  polynomials are  closed Euclidean balls.

In Section \ref{stabilt} we prove, by similar methods, that projective spaces of real stable polynomials are homeomorphic to closed Euclidean balls. This refines a result of Nuij \cite{Nuij} who proved that such spaces are contractible.

In the final section we identify topics for further studies. 
\section{Lorentzian polynomials and the symmetric exclusion process}\label{bulk}
 Let $\mathrm{H}_n^d$ denote the linear space of  all degree $d$ homogeneous polynomials in $\RR[w_1,\ldots,w_n]$, adjoin the identically zero polynomial. 
 A polynomial $f$ in $\mathrm{H}_n^d$, where $d\geq 2$, is \emph{strictly Lorentzian} if
\begin{enumerate}
\item all coefficients of $f$ are positive, and 
\item for any $1\leq i_1, i_2, \ldots, i_{d-2} \leq n$, the quadratic $\partial_{i_1} \partial_{i_2}\cdots \partial_{i_{d-2}} f$, where $\partial_j= \partial /\partial w_j$,  has the \emph{Lorentzian signature} $(+,-,-,\ldots,-)$. 
\end{enumerate}
Also, any homogeneous  polynomial of degree $0$ or $1$ with only positive coefficients is defined to be strictly Lorentzian. Polynomials that are limits  in $\mathrm{H}_n^d$, in the standard Euclidean topology on the space of coefficients of the polynomials, of strictly Lorentzian polynomials are called \emph{Lorentzian}. 

Denote by $\mathrm{L}^d_n$ the space of Lorentzian polynomials in $\mathrm{H}_n^d$, and by $\mathbb{P}\mathrm{L}^d_n$ the projectivization of $\mathrm{L}^d_n$. Since Lorentzian polynomials have nonnegative coefficients, we may identify  $\mathbb{P}\mathrm{L}^d_n$ with the space of all polynomials $f$ in $\mathrm{L}^d_n$ such that $f(\one)=1$, where $\one =(1,1,\ldots, 1)$. Let $\mathrm{E}_n^d$ be the affine Euclidean space of all $f$ in $\mathrm{H}_n^d$ for which $f(\one)=1$. Let further $\underline{\mathrm{E}}_n^d$ be the space of all \emph{multiaffine} polynomials in  $\mathrm{E}_n^d$, i.e., those polynomials in $\mathrm{E}_n^d$ that have degree at most one in each variable. Denote by 
$\mathbb{P}\underline{\mathrm{L}}_n^d$ the space of multiaffine polynomials in $\mathbb{P}\mathrm{L}^d_n$, i.e., the space of all Lorentzian polynomials in $\underline{\mathrm{E}}_n^d$. We equip   $\mathbb{P}\mathrm{L}^d_n$ and $\mathbb{P}\underline{\mathrm{L}}_n^d$ with the standard topology on the affine Euclidean spaces $\mathrm{E}_n^d$ and $\underline{\mathrm{E}}_n^d$.

For $0\leq d \leq n$, let $[n]=\{1,2,\ldots, n\}$ and  $\binom {[n]} d = \{ S \subseteq [n] \mid |S|=d\}$.   The $d$th \emph{elementary symmetric polynomial} in the variables $w_1,\ldots, w_n$ is
$$
e_d(\ww)=e_d(w_1,\ldots, w_n)= \sum_{S \in \binom {[n]} d } \ww^S, \ \ \ \ \ \mbox{ where } \ww^S= \prod_{i \in S}w_i.
$$  
Let $\SS^{n-1}=\{\xx \in \RR^n \mid x_1^2+\cdots+x_n^2=1\}$ denote the unit sphere in $\RR^n$, and let 
$\SS^{n-2}_{\one}=\{\xx \in \SS^{n-1} \mid  x_1+\cdots+x_n=0\}$. 
By e.g. \cite[Theorem 2.25, Lemma 2.5]{BH}, it follows that a multiaffine polynomial $$f= \sum_{S \in \binom {[n]} d } a(S) \ww^S$$ in $\mathrm{E}_n^d$, with nonnegative coefficients $a(S)$, is in the interior of $\mathbb{P}\underline{\mathrm{L}}^d_n$ if and only if $a(S)>0$ for all $S\in  \binom {[n]} d$, and one (and then both) of the following two conditions are satisfied
\begin{enumerate}
\item for each set $S \subset [n]$ of size $d-2$, the quadratic  $\partial^S f=\prod_{j\in S}\partial_jf$, considered as a polynomial in the $n-d+2$ variables $\{w_j \mid   j  \in [n] \setminus S\}$, has signature $(+,-,-,\ldots,-)$.
\item  for each set $S \subset [n]$ of size $d-2$ and each $\yy \in \RR^{n-d+2}$ not parallel to $\one$, the degree two polynomial 
$
\partial^Sf(t\one -\yy)
$ has two real and distinct zeros. 
\end{enumerate}
By homogeneity, (2) can be replaced by 
\begin{itemize}
\item[(2')]  for each set $S \subset [n]$ of size $d-2$ and each $\yy \in \mathbb{S}^{n-d}_\one$, the degree two polynomial 
$
\partial^Sf(t\one -\yy)
$ has two real and distinct zeros. 
\end{itemize}

\begin{lemma}\label{elementary}
Let $0\leq d \leq n$ be integers. The normalized elementary symmetric polynomial $e_d(\ww)/\binom n d$ lies in the interior of $\mathbb{P}\underline{\mathrm{L}}^d_n$. 
\end{lemma}
\begin{proof} We may assume $d\geq 2$. 
 Let $|S|=d-2$. The Hessian of $\partial^S e_d=e_2$, considered as a polynomials in $n-d+2$ variables, is $\one^T \one-I$, where $I$ is the identity matrix. This matrix has eigenvalues $-1, -1,\ldots, -1$ and $n-d+1$. 
\end{proof}

The symmetric exclusion process (SEP) is one of the main models in Interacting Particle Systems \cite{Li}. It models particles moving on a finite or countable set in a continuous way. Particles can only jump to vacant sites, and there can be at most one particle per site, see \cite{BBL, Li}. Here we view SEP as a flow $T_s : \underline{\mathrm{H}}_n^d \to \underline{\mathrm{H}}_n^d$, $s \in \RR$, a family of linear operators indexed by a real parameter (time) satisfying $$T_{s} \circ T_t=T_{s+t}, \ \ \ \mbox{ for all } s,t \in \RR, $$ 
where $T_0=I$ is the identity operator.

The symmetric group $\sym_n$ acts linearly on $\underline{\mathrm{H}}_n^d$ by permuting the variables, 
$$
\sigma(f)(\ww) = f(w_{\sigma(1)},\ldots, w_{\sigma(n)}), \ \ \mbox{ for all } \sigma \in \sym_n. 
$$ 
The flow $T_s : \underline{\mathrm{H}}_n^d \to \underline{\mathrm{H}}_n^d$, $s \in \RR$, is defined by
$ T_s =  e^{-s} e^{s\LL}$, where  
$$
\LL= \sum_{\tau} q_{\tau}\tau,   
$$
where the sum is over all transposition $\tau$ in $\sym_n$. We require that the numbers $q_{\tau}$ are nonnegative and sum to one, and that $\{\tau \mid q_\tau >0\}$ generates $\sym_n$.  Clearly $T_s: \underline{\mathrm{E}}_n^d\to \underline{\mathrm{E}}_n^d$ for all $s \in \RR$,  and  $e_d(\ww)/\binom n d$ is an eigenvector/polynomial of $\LL$.

Notice that the matrix corresponding to a transposition $\tau$, with respect to the basis $\{\ww^S \mid S \in \binom {[n]} d \}$ of $\underline{\mathrm{H}}_n^d$, is symmetric. Hence so is the matrix corresponding to $\LL$.   
Thus  $\LL$ 
 has  an orthogonal basis of eigenvectors 
$f_0,f_1, \ldots, f_N$ in $\underline{\mathrm{H}}_n^d$, where $N= \binom n d-1$ and $f_0(\ww)= e_d(\ww)/\binom n d$. Also, by the assumption on $\{ \tau \mid q_\tau >0\}$, there exists a positive number $m$ such that all entries of the matrix representing $\LL^m$ are positive.  
By Perron-Frobenius theory \cite[Chapter~1]{BP}, the eigenvalues of $\LL$ satisfy $$1=\lambda_0>\lambda_1\geq \cdots \geq \lambda_N>-1.$$
Hence if we write an element $f$ in  $\underline{\mathrm{H}}_n^d$ as $f= x_0f_0 + \sum_{j=1}^N x_j f_j$, then 
\begin{equation}\label{form}
T_s(f)= x_0f_0 + \sum_{j=1}^N x_j e^{-s(1-\lambda_j)}f_j. 
\end{equation}
Consequently $T_s(f) \to x_0f_0$ as $n \to \infty$, and since $T_s(f)(\one)=f(\one)$ we have $x_0=f(\one)$ and $f_j(\one)=0$ for all $j>0$.

Define a ball of radius $r$, centered at $f_0$,  by 
$$
B_r= \left\{ f_0 +\sum_{j=1}^N x_j f_j \mid \xx \in \RR^N \mbox{ and } \sum_{j=1}^N x_j^2 \leq r^2\right\}.
$$
Thus $B_r$ is a closed Euclidean ball in the affine space $\underline{\mathrm{E}}_n^d$. 

Using  \eqref{form}, we may read off how $T_s$ deforms the ball $B_r$. 
\begin{lemma}\label{shrink}
If $r$ and $s$ are positive numbers, then 
$$
B_{r_N(s)} \subseteq T_s(B_r) \subseteq B_{r_1(s)} \ \ \mbox{ and } \ \ B_{r_1(-s)} \subseteq T_{-s}(B_r) \subseteq B_{r_N(-s)},
$$
where $r_N(s)= r\exp(-s(1-\lambda_N))$ and $r_1(s)=r\exp(-s(1-\lambda_1))$. 
\end{lemma}

 The next proposition was essentially proved in \cite{BH}. For completeness, we provide the remaining details here. 

\begin{proposition}\label{LorSEP}
If $s\geq 0$, then $T_s : \underline{\mathrm{L}}^d_n \to \underline{\mathrm{L}}^d_n$. 
\end{proposition}

\begin{proof}
Assume first that $\LL=\tau$ is a transposition. It was proved in  \cite[Corollary 3.9]{BH} that any operator of the form 
$(1-\theta)I + \theta \tau$, where $0\leq \theta \leq 1$, preserves the Lorentzian property. Since 
$$
e^{s\LL}= \frac {e^s+e^{-s}} 2 I+ \frac {e^s-e^{-s}} 2 \tau, 
$$
 we deduce $e^{s\LL} : \underline{\mathrm{L}}^d_n \to \underline{\mathrm{L}}^d_n$ for all $s\geq 0$. 

Suppose $\LL_i : \underline{\mathrm{H}}^d_n \to \underline{\mathrm{H}}^d_n$ are linear operators such that $e^{s\LL_i} : \underline{\mathrm{L}}^d_n \to \underline{\mathrm{L}}^d_n$ for $i=1,2$ and all $s\geq 0$. 
By the Trotter product formula \cite[page 33]{EK}, 
$$
e^{s(\LL_1+\LL_2)}= \lim_{k\to \infty} \left( e^{(s/k)\LL_1} \circ e^{(s/k)\LL_2}\right)^k.  
$$
Hence $e^{s(\LL_1+\LL_2)} : \underline{\mathrm{L}}^d_n \to \underline{\mathrm{L}}^d_n$, for all $s\geq 0$, since $\underline{\mathrm{L}}^d_n$ is closed. Iterating this proves the proposition. 
\end{proof}

\begin{lemma}\label{push}
If $f \in \mathbb{P}\underline{\mathrm{L}}^d_n$ and $s>0$, then $T_s(f)$ lies in the interior of  $\mathbb{P}\underline{\mathrm{L}}^d_n$.
\end{lemma}
\begin{proof}
By the open mapping theorem and Proposition \ref{LorSEP}, $T_s$ maps the interior of  $\underline{\mathrm{L}}^d_n$ to itself for each $s\geq 0$. Hence it suffices to prove that if $f$ is in $\mathbb{P}\underline{\mathrm{L}}^d_n$, then $T_s(f)$ lies in the interior of $\mathbb{P}\underline{\mathrm{L}}^d_n$ for all $s>0$ sufficiently small.

Since there is a positive number $m$ such that all entries of the matrix representing $\LL^m$ are positive, 
 all coefficients of $T_s(f)$ are positive for all $s>0$ and $f \in \mathbb{P}\underline{\mathrm{L}}^d_n$. 

Let $f \in  \mathbb{P}\underline{\mathrm{L}}^d_n$ and suppose $S \subset [n]$ has size $d-2$, and that $\yy \in  \SS^{n-d}_\one$. Denote by $\Delta_S(\yy;s)$, the discriminant of the degree $2$ 
polynomial 
\begin{equation}\label{stsf}
t \mapsto \partial^S T_s(f)(t\one-\yy). 
\end{equation}
From \eqref{form} it follows that $\Delta_S(\yy;s)$ defines an entire function in the (complex) variable $s$. Hence either there is a positive number $\delta(\yy)$ for which  $\Delta_S(\yy;s) \neq 0$ whenever $0<|s|<\delta(\yy)$, or $\Delta_f(\yy;s)$ is identically zero. The latter can't happen since then, by letting $s \to \infty$, it follows that the discriminant of the polynomial $t \mapsto e_2(t\one -\yy)$ is equal to zero, which contradicts the fact that $e_2(\ww)$ is in the interior of $\mathbb{P}\mathrm{L}^2_n$.  By compactness (of $\SS^{n-d}_\one$) and Hurwitz' theorem on the continuity of zeros, it follows that there is a uniform $\delta>0$ such that $\Delta_S(\yy;s) \neq 0$ whenever $0<|s|<\delta$ and $\yy \in \SS^{n-d}_\one$, and $S \subset [n]$ has size $d-2$. Hence $T_s(f)$ is in the interior of $\mathbb{P}\mathrm{L}^d_n$  for all $0<s<\delta$.
\end{proof}

\section{Balls of Lorentzian polynomials}\label{proof}
Let $\A= \{A_i\}_{j=1}^m$ be a  partition of the set of variables $\{w_1,\ldots, w_n\}$, i.e., a collection of pairwise disjoint and nonempty sets such that 
$$
A_1\cup A_2 \cup \cdots \cup A_m=\{w_1,\ldots, w_n\}. 
$$
Let further $\underline{\mathrm{E}}_\A^d$ be the space of all polynomials in  $\underline{\mathrm{E}}_n^d$ that are symmetric in the variables in $A_j$ for each $j \in [m]$.

Consider the symmetric exclusion process, $T_s$, with rates $q_\tau=1/\binom n 2$  for all transpositions $\tau$ in $\sym_n$. Then $\tau \LL \tau = \LL$ for all transpositions $\tau$, so that $T_s : \underline{\mathrm{E}}_\A^d \to \underline{\mathrm{E}}_\A^d$ for all $s \in \RR$. Let $\mathbb{P}\underline{\mathrm{L}}_\A^d= \mathbb{P}\underline{\mathrm{L}}_n^d \cap  \underline{\mathrm{E}}_\A^d$.

We are now in a position to prove our main results for Lorentzian polynomials. The final tool needed is a general construction in \cite{GKL}. Explicit homeomorphisms may be extracted from the proof of \cite[Lemma~2.3]{GKL}.

\begin{theorem}\label{mainMA}
Let $\A$ be a partition of $\{w_1,\ldots, w_n\}$. The space $\mathbb{P}\underline{\mathrm{L}}^d_\A$  is homeomorphic to a closed Euclidean ball.  
\end{theorem}

\begin{proof}
Notice that $\underline{\mathrm{E}}_\A^d$ is  the intersection of $\underline{\mathrm{E}}_n^d$ with an affine space, $f_0+U$, where $U$ is a linear subspace of $\span\{f_1,\ldots, f_N\}$.  
We identify the affine linear space $\underline{\mathrm{E}}_\A^d$ with $\RR^M$, where $M=\dim(U)$, equipped with the Euclidean norm, inherited by $\span\{f_1,\ldots, f_N\}$, 
$$
\| f \|  = \sqrt{\sum_{i=1}^N x_i^2}, \ \ \ \ \ \mbox{ if } f = \sum_{i=1}^N x_if_i. 
$$ 
Consider the map 
$F : \RR \times \RR^M \to \RR^M$ defined by 
$$
F(s, f) = T_s(f). 
$$
Then 
\begin{enumerate}
\item the map $F$ is continuous, by e.g. the explicit form \eqref{form}, 
\item $F(0,f)=f$ and $F(s_1+s_2, f)=F(s_1, F(s_2,f))$, for all $f \in \RR^M$ and $s_1, s_2 \in \RR$, 
\item $\|F(s,f)\| <\|f\|$, for all $f \neq 0$ and $s>0$,  by Lemma~\ref{shrink}.
\end{enumerate}
Thus $F$ is a contractive flow, as defined in \cite{GKL}. By Lemma \ref{push}, 
$$
F(s, \mathbb{P}\underline{\mathrm{L}}^d_\A) \subset \mathrm{int}(\mathbb{P}\underline{\mathrm{L}}^d_\A) \ \ \mbox{ for all } s>0.
$$
where $\mathrm{int}$ denotes the interior. The theorem now follows from \cite[Lemma~2.3]{GKL}. 
\end{proof}

The case of Theorem \ref{mainMA}  when $\A$ is the partition $\{w_1\}, \{w_2\}, \ldots, \{w_n\}$, yields the following theorem.
\begin{theorem}\label{MA}
The space  $\mathbb{P}\underline{\mathrm{L}}^d_n$ is homeomorphic to a closed Euclidean ball. 
\end{theorem}
Let $\kappa=(\kappa_1, \ldots, \kappa_n)$ be a vector of positive integers. Define $\mathrm{E}^d_\kappa$ to be the space of all polynomials $f$ in $\mathrm{E}^d_n$ for which the degree of $f$ in the variable $w_i$ is at most $\kappa_i$, for each $1\leq i \leq n$. Also, let $\mathbb{P}\mathrm{L}^d_\kappa=\mathrm{E}^d_\kappa \cap \mathbb{P}\mathrm{L}^d_n$. 

\begin{theorem}\label{degreebounds}
The space $\mathbb{P}\mathrm{L}^d_\kappa$  is homeomorphic to  a closed Euclidean ball.
\end{theorem}
\begin{proof}
Consider the variables $w_{ij}$, $1\leq i \leq n$ and $1\leq j \leq \kappa_j$, and let $\A=\{A_1, \ldots, A_n\}$ be the partition of $\{w_{ij}\}$ given by $A_i=\{ w_{ij} \mid 1\leq j \leq \kappa_j\}$ for $1 \leq i \leq n$. 
 
The \emph{polarization operator} $\Pi^{\uparrow} : \mathrm{E}^d_\kappa \rightarrow \underline{\mathrm{E}}^d_\A$ is defined as follows. Let $f \in  \mathrm{E}^d_\kappa$. For all $i \in [n]$ and each monomial $w_1^{\alpha_1} \cdots w_n^{\alpha_n}$ in the expansion of $f$, replace $w_i^{\alpha_i}$ with 
$$
e_{\alpha_i}(w_{i1},\ldots, w_{i\kappa_i} )/ \binom {\kappa_i} {\alpha_i}.
$$
The polarization operator is a linear isomorphism with inverse $\Pi^{\downarrow}$ defined by the change of variables $w_{ij} \to w_i$. Also $\Pi^{\uparrow}$ and $\Pi^{\downarrow}$  preserve the Lorentzian property \cite[Proposition~3.1]{BH}. Hence $\Pi^{\uparrow}$ defines a homeomorphism  between $\mathbb{P}\mathrm{L}^d_\kappa$  and $\mathbb{P}\underline{\mathrm{L}}_\A^d$. The theorem now follows from Theorem~\ref{mainMA}. 
\end{proof}
By choosing $\kappa=(d,d,\ldots, d)$ in Theorem~\ref{degreebounds}, we have arrived at a proof of \cite[Conjecture~2.28]{BH}. 
\begin{theorem}\label{main}
The space $\mathbb{P}\mathrm{L}^d_n$ is homeomorphic to a closed Euclidean ball.
\end{theorem}

\section{Balls of real stable polynomials}\label{stabilt}
A polynomial $f$ in $\RR[w_1, \ldots, w_n]$ is \emph{real stable} if $f(\ww) \neq 0$ whenever $\Im(w_j)>0$ for all $j \in [n]$, see \cite{BB1,BB2,Wag}. Homogeneous real stable polynomials with nonnegative coefficients are Lorentzian, see \cite{BH}. Let $\mathrm{S}^d_n$ be the space of degree $d$ homogeneous and real stable polynomials $f$ in $ \RR[w_1, \ldots, w_n]$ with nonnegative coefficients. Let further $\mathbb{P}\mathrm{S}^d_n$ be the set of all $f \in \mathrm{S}^d_n$ for which $f(\one)=1$. It follows from \cite[Lemma~1.5]{BB1} that a polynomial $f \in \mathrm{H}^d_n$ with nonnegative coefficients is real stable if and only if for each $\yy  \in \RR^{n}$,  the polynomial 
$$
t \mapsto f(t\one -\yy)
$$
has only real zeros. Notice that by homogeneity this condition is equivalent to the condition: for each $\yy  \in \SS^{n-2}_\one$, the polynomial 
$
t \mapsto f(t\one -\yy)
$
has only real zeros. By compactness it follows that $f \in \mathrm{H}^d_n$ is in the interior of $\mathrm{S}^d_n$ if and only if all coefficients of $f$ are positive, and for each $\yy  \in \SS^{n-2}_\one$, the polynomial 
$
t \mapsto f(t\one -\yy)
$
has only real and distinct zeros.

\begin{lemma}\label{elementary-stu}
Let $0\leq d \leq n$ be integers, and let $\yy \in \RR^n$ be such that at most $n-d+1$ of the coordinates agree. The zeros $e_d(t\one -\yy)$ are real and distinct. 
\end{lemma}

\begin{proof}
For $\yy \in \RR^n$, let $p(t)= \prod_{i=1}^n (t-y_i)$. Notice that $$(n-d)!e_d(t\one -\yy) = p^{(n-d)}(t),$$ the $(n-d)$th derivative of $p(t)$. Hence all zeros of $e_d(t\one -\yy)$ are real. Moreover if $\alpha$ is a zero of $e_d(t\one -\yy)$ of multiplicity $k \geq 2$, then $\alpha$ is a zeros of $p$ of multiplicity $k+(n-d)$. The lemma follows. 
\end{proof}

For $1\leq d \leq n$, let $f_n^d(w_1,\ldots, w_n)$ be the polynomial obtained from  $e_{d}(w_{11}, w_{12}, \ldots, w_{nd})/\binom {nd} d$ by the change of variables  $w_{ij} \to w_i$ for all $1 \leq i \leq n$ and $1 \leq j \leq d$. 

\begin{lemma}\label{fnd}
The polynomial $f_n^d(\ww)$ lies in the interior of  $\mathbb{P}\mathrm{S}^d_n$. 
\end{lemma}

\begin{proof}
All coefficients of $f_n^d(\ww)$ are positive. 
We need to prove that the zeros of $f_n^d(t\one -\yy)$ are real and distinct whenever $\yy \in \RR^n$ is not parallel to $\one$. This follows from Lemma \ref{elementary-stu}, since 
$$f_n^d(t\one -\yy)= e_{d}(t\one -\widetilde{\yy})/\binom {nd} d, \mbox{ where } \widetilde{\yy}=(y_1,\ldots, y_1, \ldots, y_n, \ldots, y_n).$$ 
\end{proof}

\begin{theorem}\label{mainS}
The space   $\mathbb{P}\mathrm{S}^d_n$ is homeomorphic to a closed Euclidean ball.
\end{theorem}

\begin{proof}
The  proof is a modification of Theorem~\ref{mainMA}. 

It was proved in \cite{BBL} that SEP preserves stability on multiaffine polynomials. Also, in \cite{BB1} it was proved that the polarization operator $\Pi^{\uparrow}$ preserves stability. For $s \in \RR$, define $\widehat{T}_s : \mathrm{E}_n^d \to \mathrm{E}_n^d$ by $\widehat{T}_s= \Pi^{\downarrow} \circ T_s  \circ \Pi^{\uparrow}$, where $T_s$ is the SEP on $\underline{\mathrm{E}}_{nd}^d$, with rates $q_\tau = 1/ \binom {nd} 2$ for all transpositions $\tau$. By the discussion in Section \ref{proof}, it follows that $\widehat{T}_s$ is a contractive flow, which contracts the space to the point $f_n^d(\ww)$ in its interior (Lemma \ref{fnd}). To finish the proof, it remains to prove that  for each fixed $f \in \mathbb{P}\mathrm{S}^d_n$, there is a $\delta>0$ such that $\widehat{T}_s(f)$ is in the interior of $\mathbb{P}\mathrm{S}^d_n$  for all $0<s<\delta$.

Let $f \in \mathbb{P}\mathrm{S}^d_n$ and $\yy \in \SS^{n-2}_\one$. Denote by  $\Delta_f(\yy;s)$ the discriminant of the polynomial  
$
t \mapsto \widehat{T}_s(f)(t\one-\yy)$. 
From \eqref{form} it follows that $\Delta_f(\yy;s)$ defines an entire function in the (complex) variable $s$. Hence either there is a positive number $\delta(\yy)$ for which  $\Delta_f(\yy;s) \neq 0$ whenever $0<|s|<\delta(\yy)$, or $\Delta_f(\yy;s)$ is identically zero. The latter can't happen since then, by letting $s \to \infty$, it follows that the discriminant of the polynomial $t \mapsto f_n^d(t\one -\yy)$ is equal to zero, which contradicts the fact that $f_n^d(\ww)$ is in the interior of $\mathbb{P}\mathrm{S}^d_n$.  By compactness (of $\SS^{n-2}_\one$) and Hurwitz' theorem on the continuity of zeros, it follows that there is a uniform $\delta>0$ such that $\Delta_f(\yy;s) \neq 0$ whenever $0<|s|<\delta$ and $\yy \in \SS^{n-2}_\one$. Hence $\widehat{T}_s(f)$ is in the interior of $\mathbb{P}\mathrm{S}^d_n$  for all $0<s<\delta$.
\end{proof}

\section{Discussion}
The support of a polynomial $\sum_{\alpha \in \NN^n} a_\alpha w_1^{\alpha_1} \cdots w_n^{\alpha_n}$ is 
$\mathrm{J}= \{ \alpha \mid a_\alpha \neq 0\}$. 
There is a one-to-one correspondence between the supports of polynomials in $\mathbb{P}\mathrm{L}^d_n$ and $\mathrm{M}$-convex sets in $\Delta_n^d=\{\alpha \in \NN^n \mid \alpha_1+\cdots+\alpha_n=d\}$, see \cite{BH}. Let $\mathbb{P}\mathrm{L}_\mathrm{J}$ be the space of polynomials in $\mathbb{P}\mathrm{L}^d_n$ with support $\mathrm{J}$. 
Hence 
$$
\mathbb{P}\mathrm{L}^d_n = \bigsqcup_{\mathrm{J}}\mathbb{P}\mathrm{L}_\mathrm{J},
$$
where the disjoint union is over all $\mathrm{M}$-convex sets in $\Delta_n^d$. The space $\mathbb{P}\mathrm{L}_\mathrm{J}$ is nonempty and contractible for each $\mathrm{J}$ \cite[Theorem~3.10 and Proposition~3.25]{BH}. 

Similarly, there is a one-to-one correspondence between the supports of polynomials in $\mathbb{P}\underline{\mathrm{L}}^d_n$ and rank $d$ matroids on $[n]$. Let  $\mathbb{P}\underline{\mathrm{L}}_\mathrm{M}$ be the space of polynomials in  $\mathbb{P}\underline{\mathrm{L}}^d_n$  whose support is the set of bases of $\mathrm{M}$. Then 
$$
\mathbb{P}\underline{\mathrm{L}}^d_n = \bigsqcup_{\mathrm{M}}\mathbb{P}\underline{\mathrm{L}}_\mathrm{M},
$$
where the  union is over all matroids of rank $d$ on $[n]$. The space $\mathbb{P}\underline{\mathrm{L}}_\mathrm{M}$ is nonempty and contractible for each $\mathrm{M}$, \cite[Theorem~3.10 and Proposition~3.25]{BH}. 

Postnikov \cite{Pos} proved that there is a similar stratification of the \emph{totally nonnegative Grassmannian} into cells corresponding to matroids that are realizable by real matrices whose maximal minors are all nonnegative. He conjectured that these cells are homeomorphic to closed Euclidean balls. The conjecture was recently proved by Galashin, Karp and  Lam \cite{GKL,GKL2,GKL3} in a more general setting. The next two questions are Lorentzian analogs of Postnikov's conjecture. 
\begin{question}\label{con1}
Let $\mathrm{J}$ in $\Delta_n^d$ be an $\mathrm{M}$-convex set.  Is the closure of $\mathbb{P}\mathrm{L}_\mathrm{J}$  homeomorphic to a closed Euclidean ball? 
\end{question}
In this paper we settled the case of Question \ref{con1} when $\mathrm{J}=\Delta_n^d$.  
\begin{question}\label{con2}
Let $\mathrm{M}$ be a matroid on $[n]$. Is the closure of $\mathbb{P}\underline{\mathrm{L}}_\mathrm{M}$  homeomorphic to a closed Euclidean ball? 
\end{question}
In this paper we settled the case of Question \ref{con2} when  $\mathrm{M}$ is uniform.

We also have a stratification 
$$
\mathbb{P}\mathrm{S}^d_n = \bigsqcup_{\mathrm{J}}\mathbb{P}\mathrm{S}_\mathrm{J},
$$
where the union is 
over all $\mathrm{M}$-convex sets in $\Delta_n^d$ that occur as supports of polynomials in $\mathbb{P}\mathrm{S}^d_n$.  Are the cells $\overline{\mathbb{P}\mathrm{S}_\mathrm{J}}$ homeomorphic to Euclidean balls. Are they contractible? Similar questions could be asked for the analogous spaces of multiaffine homogeneous and real stable polynomials. Let $\mathbb{P}\underline{\mathrm{S}}^d_n$ be the space of all multiaffine polynomials in $\mathbb{P}\mathrm{S}^d_n$. In particular 
\begin{conjecture}
The space $\mathbb{P}\underline{\mathrm{S}}^d_n$ is homeomorphic to a closed Euclidean ball. 
\end{conjecture}


\begin{thebibliography}{99}

\bibitem{ALOGV} N.~Anari, K.~Liu, S.~Oveis Gharan, C.~Vinzant, Log-Concave Polynomials II: High-Dimensional Walks and an FPRAS for Counting Bases of a Matroid, Invited to Theory of Computing, \url{arXiv:1811.01816}

\bibitem{ALOGV3} N.~Anari, K.~Liu, S.~Oveis Gharan, C.~Vinzant,
Log-Concave Polynomials III: Mason's Ultra-Log-Concavity Conjecture for Independent Sets of Matroids, \url{arXiv:1811.01600}.


\bibitem{BP} A.~Berman, R.~J.~Plemmons, Nonnegative Matrices in the Mathematical Sciences, Classics Appl. Math. 9, Society for Industrial and Applied Mathematics (SIAM), Philadelphia, PA, 1994.

\bibitem{BB1} J. Borcea, P. Br\"and\'en,  {The Lee-Yang and P\'olya-Schur programs. I. Linear operators preserving stability,} Invent. Math. {\bf 177} (2009), 541--569.   

\bibitem{BB2} J. Borcea, P. Br\"and\'en,  {The Lee-Yang and P\'olya-Schur programs. II. Theory of real stable polynomials and applications,} Comm. Pure Appl. Math. {\bf 62} (2009), 1595--1631.

\bibitem{BBL} J.~Borcea, P.~Br\"and\'en,  T.~M.~Liggett, {Negative dependence and the geometry of polynomials}, J. Amer.
Math. Soc. {\bf 22} (2009), 521-567. 

\bibitem{BH} P.~Br\"and\'en, J.~Huh, {Lorentzian polynomials}, Ann. of Math. {\bf 192} (2020), 821--891.



\bibitem{EK} S.~N.~Ethier, T.~G.~Kurtz, Markov processes: Characterization and convergence.,Wiley, 1986.

\bibitem{GKL} P.~Galashin, S.~Karp, T.~Lam, The totally nonnegative Grassmannian is a ball,  \url{arXiv:1707.02010}.

\bibitem{GKL2} P.~Galashin, S.~Karp, T.~Lam, The totally nonnegative part of $G/P$ is a ball, 
Adv. Math. {\bf 351} (2019), 614--620. 

\bibitem{GKL3} P.~Galashin, S.~Karp, T.~Lam, Regularity theorem for totally nonnegative flag varieties, \url{arXiv:1904.00527}. 

\bibitem{Gur} L.~Gurvits, On multivariate Newton-like inequalities, in Advances in Combinatorial Mathematics, Springer, Berlin, 2009, pp. 61--78.

\bibitem{Li} T.~M.~Liggett, Interacting Particle Systems, Springer, 1985. 


\bibitem{MSS} A.~Marcus, D.~Spielman, N.~Srivastava, Ramanujan Graphs and the Solution of the Kadison-Singer Problem, Proc. ICM 2014.

\bibitem{Mu} K.~Murota, Discrete Convex Analysis, SIAM Monogr. Discrete Math. Appl., Philadelphia, PA, 2003. 

\bibitem{Nuij} W.~Nuij, 
			{A note on hyperbolic polynomials},  
			Math. Scand. {\bf 23} (1968), 69--72. 

\bibitem{Pos} A.~Postnikov, Total positivity, Grassmannians, and networks, \url{arXiv:math/0609764}. 

\bibitem{Wag} D.~G.~Wagner, Multivariate real stable polynomials: theory and applications, Bull. Amer. Math. Soc. {\bf 48} (2011), 53--84. 
\end{thebibliography}
\end{document}